\documentclass[12pt]{article}
\usepackage{graphicx}
\usepackage{amsmath}
\usepackage{amsfonts}
\usepackage{amsthm}
\usepackage[T1]{fontenc}
\usepackage{url}
\usepackage{color}
\usepackage[margin=1in]{geometry}
\makeatletter
\renewcommand*\env@matrix[1][\arraystretch]{%
  \edef\arraystretch{#1}%
  \hskip -\arraycolsep
  \let\@ifnextchar\new@ifnextchar
  \array{*\c@MaxMatrixCols c}}
\makeatother
\usepackage{amssymb,bm}
\usepackage{amsmath}
\usepackage{amsthm}
\usepackage{graphicx}
\usepackage[active]{srcltx} 
\usepackage{hyperref}
\hypersetup{pdfborder=0 0 0}

\newtheorem{theorem}{{\sc Theorem}}[section]
\newtheorem{proposition}[theorem]{{\sc Proposition}}

\newtheorem{lemma}[theorem]{{\sc Lemma}}
\newtheorem{corollary}[theorem]{Corollary}
\newtheorem{remark}[theorem]{Remark}

\def\XXint#1#2#3{{\setbox0=\hbox{$#1{#2#3}{\int}$ }
\vcenter{\hbox{$#2#3$ }}\kern-.6\wd0}}

\newcommand{\Ga}{\alpha}
\newcommand{\Gb}{\beta}

\newcommand{\Gf}{\phi}

\newcommand{\Gg}{\gamma}

\bmdefine\BGa{\alpha}
\bmdefine\BGb{\beta}
\bmdefine\BGd{\delta}
\bmdefine\BGe{\epsilon}
\bmdefine\BGve{\varepsilon}
\bmdefine\BGf{\phi}
\bmdefine\BGvf{\varphi}
\bmdefine\BGg{\gamma}
\bmdefine\BGc{\chi}
\bmdefine\BGi{\iota}
\bmdefine\BGk{\kappa}
\bmdefine\BGl{\lambda}
\bmdefine\BGn{\eta}
\bmdefine\BGm{\mu}
\bmdefine\BGv{\nu}
\bmdefine\BGp{\pi}
\bmdefine\BGth{\theta}
\bmdefine\BGvth{\vartheta}
\bmdefine\BGr{\rho}
\bmdefine\BGvr{\varrho}
\bmdefine\BGs{\sigma}
\bmdefine\BGvs{\varsigma}
\bmdefine\BGt{\tau}
\bmdefine\BGj{\tau}
\bmdefine\BGu{\upsilon}
\bmdefine\BGo{\omega}
\bmdefine\BGx{\xi}
\bmdefine\BGy{\psi}
\bmdefine\BGz{\zeta}
\bmdefine\BGD{\Delta}
\bmdefine\BGF{\Phi}
\bmdefine\BGG{\Gamma}
\bmdefine\BGL{\Lambda}
\bmdefine\BGP{\Pi}
\bmdefine\BGT{\Theta}
\bmdefine\BGS{\Sigma}
\bmdefine\BGU{\Upsilon}
\bmdefine\BGO{\Omega}
\bmdefine\BGX{\Xi}
\bmdefine\BGY{\Psi}

\newcommand{\CC}{{\mathcal C}}

\newcommand{\CS}{{\mathcal S}}

\newcommand{\CW}{{\mathcal W}}

\bmdefine\BCA{{\mathcal A}}
\bmdefine\BCB{{\mathcal B}}
\bmdefine\BCC{{\mathcal C}}
\bmdefine\BCD{{\mathcal D}}
\bmdefine\BCE{{\mathcal E}}
\bmdefine\BCF{{\mathcal F}}
\bmdefine\BCG{{\mathcal G}}
\bmdefine\BCH{{\mathcal H}}
\bmdefine\BCI{{\mathcal I}}
\bmdefine\BCJ{{\mathcal J}}
\bmdefine\BCK{{\mathcal K}}
\bmdefine\BCL{{\mathcal L}}
\bmdefine\BCM{{\mathcal M}}
\bmdefine\BCN{{\mathcal N}}
\bmdefine\BCO{{\mathcal O}}
\bmdefine\BCP{{\mathcal P}}
\bmdefine\BCQ{{\mathcal Q}}
\bmdefine\BCR{{\mathcal R}}
\bmdefine\BCS{{\mathcal S}}
\bmdefine\BCT{{\mathcal T}}
\bmdefine\BCU{{\mathcal U}}
\bmdefine\BCV{{\mathcal V}}
\bmdefine\BCW{{\mathcal W}}
\bmdefine\BCX{{\mathcal X}}
\bmdefine\BCY{{\mathcal Y}}
\bmdefine\BCZ{{\mathcal Z}}

\bmdefine\Bzr{ 0}
\bmdefine\Ba{ a}
\bmdefine\Bb{ b}
\bmdefine\Bc{ c}
\bmdefine\Bd{ d}
\bmdefine\Be{ e}
\bmdefine\Bf{ f}
\bmdefine\Bg{ g}
\bmdefine\Bh{ h}
\bmdefine\Bi{ i}
\bmdefine\Bj{ j}
\bmdefine\Bk{ k}
\bmdefine\Bl{ l}
\bmdefine\Bm{ m}
\bmdefine\Bn{ n}
\bmdefine\Bo{ o}
\bmdefine\Bp{ p}
\bmdefine\Bq{ q}
\bmdefine\Br{ r}
\bmdefine\Bs{ s}
\bmdefine\Bt{ t}
\bmdefine\Bu{ u}
\bmdefine\Bv{ v}
\bmdefine\Bw{ w}
\bmdefine\Bx{ x}
\bmdefine\By{ y}
\bmdefine\Bz{ z}
\bmdefine\BA{ A}
\bmdefine\BB{ B}
\bmdefine\BC{ C}
\bmdefine\BD{ D}
\bmdefine\BE{ E}
\bmdefine\BF{ F}
\bmdefine\BG{ G}
\bmdefine\BH{ H}
\bmdefine\BI{ I}
\bmdefine\BJ{ J}
\bmdefine\BK{ K}
\bmdefine\BL{ L}
\bmdefine\BM{ M}
\bmdefine\BN{ N}
\bmdefine\BO{ O}
\bmdefine\BP{ P}
\bmdefine\BQ{ Q}
\bmdefine\BR{ R}
\bmdefine\BS{ S}
\bmdefine\BT{ T}
\bmdefine\BU{ U}
\bmdefine\BV{ V}
\bmdefine\BW{ W}
\bmdefine\BX{ X}
\bmdefine\BY{ Y}
\bmdefine\BZ{ Z}

\newcommand{\R}{\mathbb{R}}

\newcommand{\sgn}{\mathrm{sign}\,}
\newcommand{\pos}[1]{(#1)_+}

\begin{document}

\title{Rank-one convexity, polyconvexity, and extremality for three-dimensional elasticity tensors in various symmetry classes}

\author{Davit Harutyunyan\thanks{University of California Santa Barbara, harutyunyan@math.ucsb.edu}
and Gagik Amirkhanyan\thanks{Independent researcher, email: agagik@gmail.com} }
\maketitle

\begin{abstract}

It has been known that, for quadratic functions, quasiconvexity equals rank-one convexity but need not equal polyconvexity. While there is an explicit characterization of polyconvexity for quadratic energies, a similar characterization for quasiconvexity is not known in dimensions $n,N \geq 3.$ This paper is concerned with the search of extremal quasiconvex quadratic forms regarded as a linear elastic energy in dimensions $N=n=3$ in various symmetry classes of the elasticity tensor. In particular, we prove that for tensors with orthotropic symmetry, quasiconvexity implies polyconvexity, and thus the orthotropic class contains no nontrivial extremals. We show that this equivalence fails first at trigonal symmetry, where among some statemets, we provide a one-parameter family of non-polyconvex trigonal extreme rays. We also provide some general criteria for proving rank-one convexity, polyconvexity, and extremality, improving a result in [\ref{bib:Har.Hov.}].

\end{abstract}

\vspace{0.5cm}


\section{Introduction}
\setcounter{equation}{0}

Quasiconvexity plays a central role in applied mathematics since the work of Morrey [\ref{bib:Morrey.1},\ref{bib:Morrey.2}]. It is equivalent to weak lower semicontinuity, and thus existence of minimizers for integral functionals with the Lagranagian satisfying certain type of growth conditions [\ref{bib:Morrey.1},\ref{bib:Morrey.2},\ref{bib:Ace.Fus.},\ref{bib:Ball},\ref{bib:Dacorogna},\ref{bib:Mueller}]. For $n,N\in \mathbb N$ a Borel measurable and locally bounded function $f\colon \mathbb R^{N\times n}\to\mathbb R$ is quasiconvex at the matrix $\BA\in\mathbb R^{N\times n}$ if

\begin{equation}
\label{1.1}
f(\bm{\BA})\leq \int_{[0,1]^n}f(\BA+\nabla\Gf(x))dx,
\end{equation}
for all $\Gf\in W_0^{1,\infty}([0,1]^n,\mathbb R^N).$ Apparently convexity implies quasiconvexity. The rank-one convexity condition occurs naturally in the second variation of the integral functional $\int_{D}L(\nabla y(x))dx,$
reducing to a pointwise condition on the Hessian of the Lagrangian $L.$ The condition of rank-one-convexity is the convexity of the function in rank-one directions, i.e., the Borel measurable and locally bounded function $f\colon \mathbb R^{N\times n}\to\mathbb R$ is rank-one-convex, if

\begin{equation}
\label{1.2}
f(\lambda\BA+(1-\lambda)\BB)\leq \lambda f(\BA)+(1-\lambda)f(\BB),
\end{equation}
for all $\lambda\in[0,1]$ and all matrices $\BA,\BB\in\mathbb R^{N\times n}$ with $\mathrm{rank}(\BA-\BB)\leq 1.$ 
For $C^2$-regular functions $f$ the condition (\ref{1.2}) is equivalent to the Legandre-Hadamard condition [\ref{bib:VanHove.1},\ref{bib:VanHove.2},\ref{bib:Dacorogna}], and reads as
\begin{equation}
\label{1.3}
\sum_{\substack{0\leq\alpha,\gamma\leq N \\ 0\leq \beta,\delta\leq n}}\frac{\partial^2 f(\bm{\xi})}{\partial\xi_{\alpha\beta}\partial\xi_{\gamma\delta}}x_\alpha y_\beta x_\gamma y_\delta\geq 0,
\end{equation}
for all $x=(x_1,x_2,\dots,x_N)\in\mathbb R^N$ and $y=(x_1,x_2,\dots,x_n)\in\mathbb R^n.$
Quasiconvexity implies rank-one convexity [\ref{bib:Dacorogna}]. It is known that quasiconvexity is not equivalent to rank one convexity in general [\ref{bib:Sverak}]. For quadratic forms quasiconvexity and rank-one-convexity are equivalent conditions [\ref{bib:Dacorogna}], and the quasiconvexity of a quadratic form $f\colon \mathbb R^{N\times n}\to\mathbb R$ reduces to 
\begin{equation}
\label{1.4}
f(x\otimes y)\geq 0\quad\text{for all vectors} \quad x\in\mathbb R^N, y\in\mathbb R^n.
\end{equation}
In this paper, we will only be dealing with quadratic functions, thus due to the above, we will use the terms quasiconvexity and rank-one convexity interchangeably. It is known, that in linear elasticity a necessary condition for a body containing a linearly elastic homogeneous material with elasticity tensor $\CC=(C_{ijkl})\in(\mathbb R^{3})^4$ to be stable when the displacement is fixed at the boundary, is the Legendre-Hadamard condition,
which is equivalent to the rank one convexity of the quadratic form $f(\xi)=(\CC\BGx;\BGx)$ associated with $\CC$, i.e., 
\begin{equation}
\label{1.5}
 f_\CC(x\otimes y)=\sum_{i,j,k,l=1}^3x_iy_jC_{ijkl}x_ky_l\geq 0 \quad\text{for all vectors} \quad x\in\mathbb R^N, y\in\mathbb R^n.
 \end{equation}
If one has equality for some non-zero $x,y$ then shear bands can form.
The condition of polyconvexity introduced by Ball [\ref{bib:Ball}] is known to be an intermediate condition between convexity and quasiconvexity. A function $f\colon \mathbb R^{N\times n}\to\mathbb R$ is polyconvex, if there exists a convex function $g\colon\mathbb R^\Sigma\to\mathbb R$ such that $f(\BGx)=g(M_1(\BGx),\dots,M_\Sigma(\BGx)),$ where $M_i(\BGx)$ are all possible minors in the matrix $\BGx.$ While there is an explicit characterization of polyconvexity for quadratic forms, there is none for quasiconvexity except for the cases $n\leq 2$ or $N\leq 2,$ in which cases quasiconvexity is equivalent to polyconvexity [\ref{bib:Terpstra},\ref{bib:Dacorogna}]. A quadratic form $f\colon \mathbb R^{N\times n}\to\mathbb R$ is polyconvex, iff it is the sum of a convex form and a Null-Langarngian [\ref{bib:VanHove.1},\ref{bib:VanHove.2}]. In dimensions $n,N\geq 3$ it is known that quasiconxeity does not imply polyconvexity. The first such example was provided by Choi in [\ref{bib:Choi}], see also [\ref{bib:Cho.Lam.}]. 
For linear elasticity energies of the symmetric strain, the gap between the convex cones of quasiconvex and polyconvex quadratic forms is more delicate but still present: Zhang~[\ref{bib:Zhang}] gave the first symmetric rank-one-convex form that is not symmetric polyconvex, see also [\ref{bib:Bou.Kre.Sch.}]. Milton and the first author attempted to characterize all $3\times 3$ ($n=N=3$) quasiconvex quadratic forms in a series of papers [\ref{bib:Har.Mil.1},\ref{bib:Har.Mil.2},\ref{bib:Har.Mil.3}]. The approach was to characterize all so called Milton extremals, the ones that lose the quasiconvexity property upon subtraction a convex form linearly independent with the original form [\ref{bib:Milton.3}, page 87], see also [\ref{bib:Milton.1}, section 25.2]. Those extremal forms are populated in the gap between the two cones mentioned above. A complete characterization would then follow from the fact, that the convex cone of $n\times N$ quasiconvex quadratic forms satisfies a simplified Krein-Milman property, namely, any such form is the sum of exactly one Milton extremal and a polyconvex form [\ref{bib:Milton.3}]. Milton extremals were later labeled as "weak extremals" due to the fact, that in dimensions $n,N\geq 4,$ they are not necessarily extreme rays of the cone in the classical sense, see [\ref{bib:Har.Hov.}]. Whether or not weak and strong extremality are equivalent is open in the case $n=N=3.$ The square of a linear form is a trivial strong extremal; the ones that ate not squares will be called nontrivial. The partial characterization of weak extremals in [\ref{bib:Har.Mil.1},\ref{bib:Har.Mil.2},\ref{bib:Har.Mil.3}] were successfully pushed for strong extremals with orthotropic symmetry in [\ref{bib:Harutyunyan}] and for general strong extremals in [\ref{bib:Har.Hov.}]. A motivation for studying extremals is to utilize them to derive new bonds in the theory of composites. 
 In the theory of composites, one powerful method for obtaining bounds on the effective tensor has proven to be the
translation method pionered by Tartar and Murat [\ref{bib:Tartar.1},\ref{bib:Mur.Tar.},\ref{bib:Tartar.2}] and Lurie and Cherkaev [\ref{bib:Lur.Che.1}, \ref{bib:Lur.Che.2}], see also [\ref{bib:Nes.Rog.},\ref{bib:All.Koh.},\ref{bib:Che.Gib.1},\ref{bib:Che.Gib.2},\ref{bib:Mil.Che.},\ref{bib:Koh.Lip.},\ref{bib:Milton.2},\ref{bib:Mil.Ngu.},\ref{bib:Kan.Kim.Mil.},\ref{bib:Kan.Mil.Wan.}]. 
 These bounds are tightest if one uses extremal quasiconvex forms as shown in [\ref{bib:Milton.1}, page 87], see also [\ref{bib:Milton.2}, section 25.2]. Allaire and Kohn [\ref{bib:All.Koh.}]  used certain type of extremals to bound the elastic energy of two phase composites with isotropic phases. Kang and Milton [\ref{bib:Kan.Mil.}] extended the translation method as a tool for bounding the volume fractions of materials in a two-phase body from boundary
measurements and in this context too it is natural to use extremal  quasiconvex forms. Extremal quasiconvex forms are also the best choice of quasiconvex functions for obtaining series expansions for effective tensors that have an extended domain of convergence, 
and thus analyticity properties as a function of the component moduli on this domain; see section 14.8 and page 373 of section 18.2 of [\ref{bib:Milton.2}]. The focus of this paper will be to study $3\times 3$ quasiconvex quadratic forms in various symmetry classes, from the linear elasticity perspective. Recall that $n=N=3$ corresponds to three dimensional linear elasticity, where the stiffness tensor is a fourth order tensor $\CC=(C_{ijkl})\in(\mathbb R^{3})^4$ that due to the symmetry of the stress sensor, can be reduced to a sixth order symmetric matrix (the Voigt notation), the linear elasticity matrix with 21 parameters in general. The energy depends only on the strain 
$$
 \varepsilon(\BGx):=\frac12(\BGx+\BGx^T),\qquad
$$
where $\BGx=(\xi_{ii})_{i,j=1}^3$ plays the role of the deformation gradient. The energy is given by 
\begin{equation}
\label{1.6}
 f_\CC(\BGx)=\sum_{i,j,k,l=1}^3C_{ijkl}\varepsilon(\BGx)_{ij}\varepsilon(\BGx)_{kl}.
 \end{equation}
 
The main objective of this paper is the search (or characterization) of strong extremal $3\times 3$ quasiconvex quadratic forms in various symmetry classes of the linear elasticity Voigt matrix $\BC\in \mathbb R^{6\times 6},$ continuing the program of Milton and the second author. 

It was shown in [\ref{bib:Har.Mil.1}], that for linear elastic materials with cubic symmetry (this also includes isotropic materials), quasiconvexity of the energy implies polyconvexity. The next symmetry classes according to the standard classification of the crystallographic elasticity classes are the hexagonal (5 independent constants), trigonal (6 independent constants), tetragonal (6 or 7 independent constants), orthotropic (9 independent constants), monoclinic (13 independent constants), and triclinic (21 independent constants) in the increasing order. 
The orthotropic class contains all the previous classes as subclasses except the trigonal class. We will focus on the orthotropic and trigonal classes in the present paper. We will show that there are no nontrivial extremals in the orthotropic class (Theorem~\ref{Thm:3.1}), and provide a single-parameter family of strong extremals in the trigonal class (Theorem~\ref{Thm:4.4}). This implies that the first class on the classification chart containing nontrivial extremals is the trigonal class. 

The paper is organized as follows: In Section~2, we provide some simple, but useful criteria for rank-one convexity, polyconvexity, and strong extremality improving a result in [\ref{bib:Har.Hov.}]. In Section~3 we study forms with orthotropic symmetry completely, and finally in Section~4 we study forms with trigonal symmetry. The results concerning each symmetry class will be formulated and proved in the same appropriate section. Also, capital boldface letters will be used for matrices, and callographic letters will denote tensors  throughout the paper. 

It is worth mentioning that the related (same) problems have been studied in the convex geometry and real algebraic geometry communities. We refer the interested reviewers to the papers  [\ref{bib:Reznick},\ref{bib:Quarez},\ref{bib:Cho.Lam.},\ref{bib:Cho.Lam.Rez.},\ref{bib:Li.Wu.},\ref{bib:Hou.Li.Poo.Qi.Sze.},\ref{bib:BLe.Hau.Ott.Ran.Stu.},\ref{bib:Kun.Sch.}], and the references therein.

\section{Criteria for rank-one convexity, polyconvexity and extremality}
\label{Sec:2}
\setcounter{equation}{0}

In this section we provide some seemingly simple, but very useful criteria for rank-one convexity, polyconvexity and extremality. One is the so-called soft-pair rigidity route, the other one is global determinant nonnegativity coupled with maximal inertia at a point, and the last one is an add-on to the acoustic tensor extremal-determinant route for polyconvexity, provided by the first author and Hovsepyan in [\ref{bib:Har.Hov.}, Theorem~2.1].  We start with the soft-pair rigidity statements.

A \textit{soft pair} of a rank-one-convex tensor $\CC\in (\mathbb R^{n})^4$ is a pair $(x,y)\in \mathbb P^{n-1}\times \mathbb P^{n-1}$ such that $f_\CC(x\otimes y)=0$, where $\mathbb P^{n-1}$ is the unit sphere in $\mathbb R^{n}.$ The following criterion is independent of the acoustic-determinant extremality and is useful when the determinant criterion is not applicable.

\begin{proposition}[Soft-pair rank criterion I]
\label{Pro:2.1}
Let $\CC\in (\mathbb R^{3})^4$ be a rank-one convex symmetric-strain tensor with soft pairs 
$Z=\{(x_k,y_k)\in \mathbb P^2\times \mathbb P^2 \ : \ k=1,\dots,N\}.$ Let
$\CW(Z)$ be the set of all rank-one convex symmetric-strain tensors $\CC' \in(\mathbb R^{3})^4$ such that 
$\nabla f_{\CC'}(x_k\otimes y_k)=0,\ k=1,\dots,N,$ i.e.,  
$$
\CW(Z)=\; \bigl\{ \CC' \ :  \ \nabla f_{\CC'}(x_k\otimes y_k)=0,\ k=1,\dots,N\,\bigr\}
\;\subseteq\;\R^{21}.
$$
If $\dim\left(\CW(Z)\cup (-\CW(Z))\right)=1$ (necessarily $\CW(Z)\cup (-\CW(Z))=\R \CC$), then $\CC$ spans an extreme ray of the symmetric-strain rank-one convexity cone.
\end{proposition}

\begin{proof}
Let $f_\CC=f_1+f_2$ with $f_i$ rank-one convex symmetric-strain quadratic forms. At each soft pair we have
 $f_1(x_k\otimes y_k)+f_2(x_k\otimes y_k)=0$ with both terms nonnegative, so each vanishes. Hence, $(x_k,y_k)$ is a point of global minimum of the polynomial $f_i$ on all of $\R^3\times\R^3$, thus its gradient vanishes too: $\nabla f_i(x_k\otimes y_k)=0$. Thus we get 
 $\CC_{i}\in \pm\CW(Z)$ for $i=1,2,$ hence $\CC_{i}=\lambda_i \CC$ with some $\lambda_i\ge 0$ because $f,f_i\ge0.$ This yields $f_i=\lambda_i f.$
\end{proof}

Let now $\mathcal B_{3,3}$ denote the $36$-dimensional space of all real biquadratic forms on $\mathbb R^3\times\mathbb R^3$.  
For a finite collection $Z\subset \mathbb P^2\times \mathbb P^2$ as above, define similarly
$$
 \mathcal W_{\mathrm{full}}(Z):=
 \{F'\in\mathcal B_{3,3}:\nabla F'(x_k\otimes y_k)=0  \ \text{for} \ i=1,\dots,N\}.
$$
\begin{proposition}[Soft-pair rank criterion II]
\label{Pro:2.2}
Let $F\in \mathcal B_{3,3}$ be a nonnegative nonzero biquadratic and let $Z$ be any collection of its soft pairs. 
If $\dim\left(\CW_{full}(Z)\cup (-\CW_{full}(Z))\right)=1$, then $F$ spans an extreme ray of the full cone of nonnegative $3\times3$ biquadratics.
\end{proposition}

\begin{proof}
If $F=F_1+F_2$ with $F_i\ge0$, then every point of $Z$ is a global minimum of
each $F_i$, so $\nabla F_i=0$ there.  Hence
$F_i\in\pm \mathcal W_{\mathrm{full}}(Z)$, and nonnegativity forces the
multipliers to be nonnegative.
\end{proof}

The next criterion is useful in proving rank-one convexity of a quadratic form (non-negativity of a biquadratic). 
\begin{proposition}[Inertia continuation]
\label{Pro:2.3}
Let $\BM(x)\colon \mathbb P^{n-1}\to \mathbb R^{n\times n}$ be a continuous matrix field. Suppose that the zero determinant set 
$\tilde Z:=\{x\in \mathbb P^{n-1} : \det \BM(x)=0 \}$ is finite, $\det \BM(x)\ge0$ on $\mathbb P^{n-1}$, and $\BM(x_0)\succ0$ at one point $x_0\in\mathbb P^{n-1} \setminus \tilde Z$. Then $\BM(x)\succeq0$ for every $x\in \mathbb P^{n-1}$.
\end{proposition}
\begin{proof}
On $\mathbb P^{n-1}\setminus \tilde Z$ the matrices are nonsingular. Their inertia is locally 
constant by continuity,  hence globally constant because $\mathbb P^{n-1}\setminus \tilde Z$ is simply-connected. It equals
$(n,0,0)$ at $x_0$, so $\BM(x)\succ0$ throughout $ \mathbb P^{n-1}\setminus Z$. Every point of $\tilde Z$ is a limit of
points in $\mathbb P^{n-1} \setminus \tilde Z$; therefore again by continuity we get $\BM(x) \succeq 0$ on $\tilde Z$ as well.
\end{proof}

Recall, that for $\CC\in\mathbb (\mathbb R^{3})^4$ and for a quasiconvex quadratic form 
$g_\CC(\BGx)=\sum_{i,j,k,l=1}^3C_{ijkl}\xi_{ij}\xi_{kl},$ the acoustic $y$-tensor of $g$ is the $3\times 3$ matrix $\BA(y)$ such that  $g(x\otimes y)=x^T \BA(y)x.$ Also, from $g(x\otimes y)\geq 0$ one has that $\BA(y)$ is positive semidefinite for all $y\in\mathbb R^3.$
The next statement excludes nontrivial extremality in part (ii) of [\ref{bib:Har.Hov.}, Theorem~2.1].
 
\begin{theorem}[Perfect squares determinants]
\label{Thm:2.4}
Let $\CC\in\mathbb (\mathbb R^{3})^4$ and the quasiconvex quadratic form $g_\CC(\BGx)=\sum_{i,j,k,l=1}^3C_{ijkl}\xi_{ij}\xi_{kl}$ 
on $\R^{3\times3}$ be such that the determinant of its acoustic tensor is a perfect square (the zero polynomial included). Then $f$ is polyconvex. 
\end{theorem}

\begin{proof}
 Thanks to [\ref{bib:VanHove.1},\ref{bib:Terpstra}], we need to prove that the biquadratic 
 $f(x\otimes y)\geq 0$ is a sum of squares of bilinear forms. If $\det \BA(y)=h(y)^2,$ then either $h\equiv 0,$ or
 the restriction of $h(y)$ to any line passing through the origin is a real binary cubic, which either vanishes identically (so the line lies in the zero set of $h$), or has a real linear factor, hence the zero set of $h$ is infinite on the unit sphere. In either case $\det \BA(y)$ is a singular positive semidefinite matrix at infinitely many $y\in \mathbb P^2$, each contributing some $x\in \mathbb P^2$ with $f(x\otimes y)=0$. Distinct $y$ coordinates give distinct projective zeros, so $f$ has infinitely many real zeros.
By a theorem of Quarez [\ref{bib:Quarez}, Thm.~4.4], a nonnegative $(3,3)$-biquadratic with
infinitely many real zeros is a sum of squares. Consequently, $f$ is polyconvex. 
\end{proof}

\section{Forms with Orthotropic Symmetry}
\label{Sec:3}
\setcounter{equation}{0}

First we review the definition of orthotropic materials, i.e., quadratic forms that have orthotropic symmetry in linear elasticity. A homogeneous orthotropic elastic material has three mutually orthogonal planes such that the material properties are symmetric under reflection about each plane. When the Cartesian coordinate axes are chosen orthogonal to these planes, then the material properties are invariant under the reflections $x_\alpha \to -x_\alpha$, $x_\beta \to x_\beta$, and $x_\gamma\to x_\gamma$, where $\Ga\Gb\Gg$ runs over all permutations of $123$.  For orthotropic materials, the entries of the elasticity tensor such as $C_{\Ga\Gb\Gg\Gg}$ and $C_{\Ga\Gb\Gb\Gb},$ change sign under a reflection about a symmetry plane mentioned above, thus those must be zero. Thus the elements $C_{ijkl}$ of the elasticity tensor must be zero unless the indices $ijkl$ contain an even number of repetitions of the indices $1$, $2$ or $3$. This means, that an orthotropic tensor is block diagonal, consisting of a normal $3\times3$ block $\BB=(C_{ij})_{i,j\le3}$ and a diagonal shear block
$\mathrm{diag}(C_{44},C_{55},C_{66})$, with no off-block coupling. In the Voigt notation, the stress-strain relations takes the form 

\begin{equation}
\label{3.1}
\BGs=\BC\cdot\bm{\varepsilon},\quad\text{where}\quad
\BGs=\begin{bmatrix} \sigma_{11} \\ \sigma_{22} \\ \sigma_{33} \\ \sigma_{23} \\ \sigma_{31} \\ \sigma_{12} \end{bmatrix}, \quad
\bm{\varepsilon}=\begin{bmatrix} \varepsilon_{11} \\ \varepsilon_{22} \\ \varepsilon_{33} \\ 2\varepsilon_{23} \\ 2\varepsilon_{31} \\ 2\varepsilon_{12} \end{bmatrix}, \quad
\BC=
\begin{bmatrix}
C_{11} & C_{12} & C_{13} & 0 & 0 & 0\\
C_{12} & C_{22} & C_{23} & 0 & 0 & 0\\
C_{13} & C_{23} & C_{33} & 0 & 0 & 0\\
0 & 0 & 0 & C_{44} & 0 & 0\\
0 & 0 & 0 & 0 & C_{55} & 0\\
0 & 0 & 0 & 0 & 0 & C_{66}
\end{bmatrix},
\end{equation}
and linear elastic tensors with orthotropic symmetry have nine independent moduli. 
The quadratic energy is given by 

\begin{align}
\label{3.2}
f_\BC(\bm{\xi})&=\sum_{i,j=1}^3 C_{ij}\epsilon_{ii}\epsilon_{jj}+4C_{44}\epsilon_{23}^2+4C_{55}\epsilon_{31}^2+4C_{66}\epsilon_{12}^2\\ \nonumber
&=\sum_{i,j=1}^3 C_{ij}\xi_{ii}\xi_{jj}+C_{44}(\xi_{23}+\xi_{32})^2+C_{55}(\xi_{31}+\xi_{13})^2+C_{66}(\xi_{12}+\xi_{21})^2,
\end{align}
where the variable $\xi_{ij}$ plays the role of the $ij-$th entry of the displacement gradient $\nabla y=(\frac{\partial y_i}{\partial x_j}),$ $i,j=1,2,3.$ The mechanical properties of the material are in general different along each axis. Orthotropic materials require 9 elastic constants and have as subclasses all the previous classes (on the classification chart), except the trigonal class. The wood in a tree trunk is an example of a material which is locally orthotropic: the material properties in three perpendicular directions, axial, radial, and circumferential, are different. Many crystals and rolled metals are also examples of orthotropic materials.

Below is the main theorem of this section. 

\begin{theorem}
\label{Thm:3.1}
For orthotropic linear elasticity rank-one convexity, (quasiconvexity ) implies polyconvexity.
\end{theorem}

\begin{proof}
We need to prove, that given 
$$
f(\bm{\xi})=\sum_{i,j=1}^3 C_{ij}\xi_{ii}\xi_{jj}+C_{44}(\xi_{23}+\xi_{32})^2+C_{55}(\xi_{31}+\xi_{13})^2+C_{66}(\xi_{12}+\xi_{21})^2,
$$
with $f(x\otimes y)\geq 0,$ the biquadratic form $f(x\otimes y)$ is a sum of squares of linear forms in $x_iy_j$ variables. For simplicity we introduce the variables $u_i=x_iy_i,$ $i=1,2,3,$ $v_k=x_iy_j+x_jy_i,$ and $w_k=x_iy_j-x_jy_i$ for all $\{i,j,k\}=\{1,2,3\},$ and set $u=(u_1,u_2,u_3),v=(v_1,v_2,v_3),w=(w_1,w_2,w_3)\in\mathbb R^3.$ We will prove that $f(x\otimes y)$ is a sum of squares of linear forms in $u_i,v_i,$ and $w_i$ variables. 
Consecutively, we have the identity in the new variables:
$$
f(x\otimes y)=u^T\BB u+C_{44}v_1^2+C_{55}v_2^2+C_{66}v_3^2.
$$
Let $\BT(t)$ be the symmetric $3\times3$ matrix with zero diagonal entries and off-diagonals entries given by $T_{12}=t_3,\ T_{13}=t_2,\ T_{23}=t_1$. Then keeping in mind the identities $v_k^2-w_k^2=4u_iu_j$ for all $\{i,j,k\}=\{1,2,3\},$ we obtain 
for any $t=(t_1,t_2,t_3)\in\R^3$, the representation for the energy:
\begin{equation}
\label{3.3}
f(x\otimes y)=u^T[\BB+2\BT(t)]u+\sum_{k=1}^{3}(C_{k+3,k+3}-t_k)v_k^2+\sum_{k=1}^{3}t_kw_k^2
\end{equation}
Note that from $f(e_i\otimes e_j)\geq 0$ with $i\neq j,$ we have $C_{44},C_{55},C_{66}\geq0.$ Hence, if we could prove that there exists a $t\in Q=[0,C_{44}]\times[0,C_{55}]\times[0,C_{66}]$ such that the matrix $\BB+2\BT(t)$ is positive semidefinite, then (\ref{3.3}) will imply that $f( x\otimes y)$ is a sum of squares, thus polyconvex. 

Next, define the associated function $F\colon R^{3\times 3}_{sym}\to\mathbb R$ on symmetric  $3\times 3$ matrices as follows:
\begin{align}
\label{3.4}  
F(\BX)&=\langle \BB,\BX \rangle+4\big(C_{44}\pos{X_{23}}+C_{55}\pos{X_{13}}+C_{66}\pos{X_{12}}\big)\\ \nonumber
&=\sum_{i,j=1}^3B_{ij}X_{ij}+4\big(C_{44}\pos{X_{23}}+C_{55}\pos{X_{13}}+C_{66}\pos{X_{12}}\big),
\end{align}
where $\pos{x}=\max(x,0).$ Clearly, $F$ is a convex function, being a linear part plus a nonnegative combination of the
convex maps $\BX\mapsto\pos{X_{ij}}$. Consider further the spectraplex 
\begin{equation}
\label{3.5}
\CS=\{\BX\in\mathbb R^{3\times 3 } \ : \ \BX\succeq0, \  \mathrm{Tr}\BX=1\}.
\end{equation}
It is known from the linear programming and convex optimization duality theory [\ref{bib:Rockafellar}], that the extreme points of $\CS$ are the rank-one 
matrices $\BX=x\otimes x$ with $\mathrm{Tr}(x\otimes x)=|x|^2=1$. Set for the sake of simplicity
\begin{equation}
\label{3.6}
E(x):=F(x\otimes x)=x^T\BB x+4\sum_{k=1}^3C_{k+3,k+3}\pos{x_ix_j},
\end{equation}
where for each pair $(i,j),$ the index $k$ is the third one. The following is a key lemma in the proof. 

\begin{lemma}
\label{Lem:3.2}
The following statements hold.
\begin{enumerate}
\item[(i)] The energy $f$ is rank-one convex if and only if  $E(x)\ge0$ for all $x\in\R^3$ with $|x|=1,$ or 
equivalently if and only if $F(x\otimes x)\geq 0$ for all extreme points $x\otimes x\in \CS.$ 

\item[(ii)] The smallest eigenvalue of the matrix $\BB+2\BT(t)$ satisfies 
$$\max_{t\in Q} \lambda_{\min}\!\big(\BB+2\BT(t)\big)=\min_{\BX\in\CS}F(\BX).$$
Hence a feasible $t\in Q$ with $\BB+2\BT(t)$ positive semidefinite exists if and only if $F\ge0$ on all of $\CS$.
\end{enumerate}
\end{lemma}

\begin{proof}[Proof of Lemma~3.2]

\noindent Part \textbf{(i)}. Assume first $E(x)\ge0$ for all $x\in\R^3$ with $|x|=1.$ Since $E$ is homogeneous of degree two
in $x,$ then we have $E(z)\geq 0$ for all $z\in\mathbb R^3.$ From $C_{k}\ge 0$ for $k=4,5,6$ and by the obvious inequality 
$v_k^2=w_k^2+4u_iu_j\ge4\pos{u_iu_j}$, we obtain for the energy
\begin{align}
\label{3.7} 
f(x\otimes y)&=u^T\BB u+\sum_{k=1}^3C_{k+3,k+3} v_k^2\\ \nonumber
& \geq u^T\BB u+4\sum_{k=1}^3C_{k+3,k+3}\pos{u_iu_j}\\ \nonumber
&=E(u).
\end{align}
Therefore we obtain $f(x\otimes y)\geq E( u)\geq 0$ and thus $f $ is rank-one convex. Assume now $f$ is
rank-one convex. The idea is to test the inequality $f(x\otimes y)\geq 0$ with a suitably chosen rank-one matrix 
$x\otimes y.$ For any given $z\in\mathbb R^3,$ we choose $x_i=\sqrt{|z_i|}$ and $y_i=\sgn(z_i)\sqrt{|z_i|}.$ This gives $ u=z$ and $v_k=\big(\sgn(z_i)+\sgn(z_j)\big)\sqrt{|z_iz_j|},$ with the prefactor $\sgn(z_i)+\sgn(z_j)$ being  $\pm2$ when $z_i,z_j$ share a sign and $0$ otherwise. Consequently we get
$v_k^2=4\pos{z_iz_j}$ and hence $f(x\otimes y)=E(z).$ Hence $f(x\otimes y)\geq 0$ for all $x,y\in\mathbb R^3$ implies $E(z)\geq 0$ for all $z\in\mathbb R^3.$ This proves (i).\\

\noindent Part \textbf{(ii)}. For any fixed $t\in\mathbb R^3,$ we have that the minimal eigenvalue satisfies
 \begin{equation}
\label{3.8} 
\lambda_{\min}(\BB+2\BT(t))=\min_{|x|=1}x^T\left(\BB+2\BT(t)\right)x
\end{equation} 
 by the spectral theorem. Because the spectraplex $\CS$ is convex and compact with extreme points $x\otimes x\in \CS$ with $|x|=1,$ 
 the linear functional $G(\BX)=\langle \BX,\BB+2\BT(t)\rangle\colon \mathbb R^{3\times 3}_{sym}\to\mathbb R$ attains its  
 minimum on $\CS$ at an extreme point in $\CS$ by the finite dimensional Krein-Milman (Minkowski-Carath\'eodory) theorem, hence we have 
 \begin{align}
\label{3.9} 
\min_{\BX\in\CS}\langle \BX,\BB+2\BT(t)\rangle&=\min_{\Bx\otimes\Bx\in\CS}\langle \Bx\otimes\Bx,\BB+2\BT(t)\rangle\\ \nonumber
&=\min_{|x|=1}x^T\left(\BB+2\BT(t)\right)x.
\end{align}  
 Putting together (\ref{3.8}) and  (\ref{3.9}) we arrive at
 \begin{equation}
\label{3.10} 
 \lambda_{\min}(\BB+2\BT(t))=\min_{\BX\in\CS}\langle \BX,\BB+2\BT(t)\rangle.
 \end{equation}
Next, for any fixed $\BX\in\CS$ we have
\begin{equation}
\label{3.11} 
\max_{t\in Q}\langle \BX,\BB+2\BT(t)\rangle=\langle \BX,\BB\rangle+4\sum_{k=1}^3C_{k+3}\pos{X_{ij}}=F(\BX),
\end{equation}
since $t_k\in[0,C_{k+3}]$ multiplies $4X_{ij}$ and is driven to the endpoint matching $\sgn X_{ij}$. Now the function
$G(\BX)=\langle \BX,\BB+2\BT(t)\rangle$ is affine in both $\BX$ and $t,$ and $Q$ and $\CS$ are
compact convex sets, thus Sion's minimax theorem~[\ref{bib:Sion}] gives
\begin{equation}
\label{3.12} 
\max_{\Bt\in Q}\min_{\BX\in\CS}\langle \BX,\BB+2\BT(t)\rangle=\min_{\BX\in\CS}\max_{t\in Q}\langle \BX,\BB+2\BT(t)\rangle.
\end{equation} 
Finally, combining (\ref{3.10})-(\ref{3.12}), we discover 
 \begin{equation}
\label{3.13} 
 \max_{t\in Q}\lambda_{\min}(\BB+2\BT(t))=\min_{\BX\in\CS}F(\BX).
 \end{equation} 
This proves (ii) and the lemma. 
\end{proof}

Getting back to the proof of the theorem, we have by (ii) of the lemma, that it is sufficient to show that $\min_\CS F\geq 0.$ Assume in contradiction $\min_\CS F=m<0.$ We will show, that then there exists a unit vector $x_0\in\mathbb R^3$ such that 
 $E(x_0)=m<0$. Take a pair $(t^0,\BX^0)\in Q\times \CS$ from (ii) of the lemma so that 
$\BX^0\in\CS$ minimizes $F$ on $\CS$ and 
 \begin{equation}
\label{3.14} 
\langle \BX^0,\BB^0\rangle=m=\lambda_{\min}(\BB^0),\quad \text{with }\quad \BB^0:=\BB+2\BT(t^0).
 \end{equation} 
 Let $V_0$ be the least eigenspace of $\BB^0$. Equation (\ref{3.14}) implies that $\BX^0$ is supported on $V_0.$ Indeed, from 
 $\BX^0\in \CS,$ due to the finite dimensional Krein-Milman theorem, $\BX^0$ is a convex combination of the extreme points of $\CS,$ i.e.,
 \begin{equation}
\label{3.15} 
\BX^0=\sum_{k=1}^nc_k x_k\otimes x_k,\quad |x_k|=1,  \ \ c_k\geq 0,  \ \ \sum_{k=1}^n c_k=1.
\end{equation}
This implies on one hand that,
\begin{equation}
\label{3.16} 
\lambda_{\min}(\BB^0)=\langle \BX^0,\BB^0\rangle=\sum_{k=1}^nc_kx_k^T\BB^0 x_k.
\end{equation}
and on the other hand the spectral theorem yields
\begin{align}
\label{3.17} 
\sum_{k=1}^nc_kx_k^T\BB^0 x_k&\geq  \sum_{k=1}^nc_k\lambda_{\min}(\BB^0)|x_k|^2\\ \nonumber
&=\lambda_{\min}(\BB^0),
\end{align}
hence we have equality in (\ref{3.17}), i.e., each $x_k$ is in $V_0$ and so are the columns and raws of $\BX^0$. 
The pair's other half $t^0$ maximizes 
$$\langle \BX^0,\BB+2\BT(t)\rangle=\langle \BX^0,\BB\rangle+4\sum_{i<j}t_{ij}X^0_{ij}$$
over the box $Q$ maximizing each $t_{ij}\mapsto4t_{ij}X^\ast_{ij}$ on $[0,C_{k+3}]$ separately, yielding the sign conditions
\begin{equation}
\label{3.18}
  t^0_{ij}=C_{k+3}\Rightarrow X^0\ge0,\qquad
  t^0_{ij}=0\Rightarrow X^0_{ij}\le0,\qquad
  0<t^0_{ij}<C_{k+3}\Rightarrow X^0_{ij}=0,
\end{equation}
for each pair $(i,j)$ with $i<j.$ Writing 
$$b_{ij}(z):=C_{k+3}\pos{z_iz_j}-t^\ast_{ij}z_iz_j\ge0,\quad z\in\mathbb R^3,$$
we have
 $$E(z)=z^T\BB^0 z+4\sum_{i<j}b_{ij}(z),$$
so for a unit $z\in V$, where one has $z^T\BB^0 z=\lambda_{\min}(\BB^0)=m$, it suffices
to find a unit $z$ with $b_{ij}(z)=0$ for all three pairs $(i,j).$ By (\ref{3.18}) this 
reqires: $z_iz_j\ge0$ if $t^0_{ij}=C_{k+3}$, $z_iz_j\leq 0$ if $t^0_{ij}=0$, and
$z_iz_j=0$ if $0<t^0_{ij}<C_{k+3}.$ 
If $\dim V_0=1$ then $\BX^0=x^0 \otimes x^{0}$ for some $x^{0}\in V$ with $|x^0|=1$ and
$E(x^0)=F(\BX^0)=m$, which contradicts (i) of the lemma. If $\dim V_0=3,$ then $\BB^0=m I,$ thus 
 $z=e_1$ works ($z_iz_j=0$). Otherwise assume $\dim V_0=2.$ Let $n=(n_1,n_2,n_3)$ be a unit eigenvector of the other 
 eigenvalue $\lambda>m$ of $\BB^0.$ Then $V_0=n^\perp$. If some $n_k=0$, then $e_k\perp n$, so $e_k\in V_0$
and $e_k^T\BB^0 e_k=m,$ thus $z=e_k$ works. Assume finally $n_k\neq 0$ for $k=1,2,3.$ As 
$\BX^0 n=0$, we have
$$0=n^T\BX^0 n=\sum_{k=1}^3 X^0_{kk}n_k^2+2\sum_{i<j}X^0_{ij}n_in_j,$$
thus 
\begin{equation}
\label{3.19}
 \sum_{i<j}X^0_{ij}n_in_j =-\tfrac12 \sum_{k=1}^3 X^0_{kk}n_k^2\leq 0.
 \end{equation}
Note that the sum in (\ref{3.19}) is strictly negative, because equality would force $X^0_{kk}n_k^2=0$, hence
(as $n_k\ne0$) $X^0_{kk}=0$ for all $k$ which contradicts  $\operatorname{Tr}\BX^0=1$. 
Thus some pair $(i,j)$ has $X^0_{ij}n_in_j<0$, in particular $X^0_{ij}\ne 0$, so by (\ref{3.18}) that pair is pinned and
$\sgn X^0_{ij}$ is the sign required of $z_iz_j$. Let $k$ be the third index. If $V_0=\{x_k=0\},$ then $z=e_i\in V_0$ works. 
Otherwise $\dim(V_0\cap\{x_k=0\})=1,$ since two nonidentical crossing planes in $\R^3$ meet in at a straight line. 
Thus we can choose a unit vector $z\in V_0\cap\{x_k=0\}=\operatorname{span}(n\times e_k).$ Then
$z_k=0$, so $b_{ik}(z)=b_{jk}(z)=0,$ and $(n\times e_k)_i(n\times e_k)_j=-n_in_j$ gives
$\sgn(z_iz_j)=-\sgn(n_in_j)=\sgn X^0_{ij}$, the required sign, so
$b_{ij}(z)=0$. Hence $b_{ij}(z)=0$ for all pairs and $E(z)=m<0$ which contradicts the rank-one convexity of $f$
thanks to the Lemma. This completes the proof of the theorem. 

\end{proof}

\begin{corollary}
Rank-one convexity implies polyconvexity in all symmetry subclasses of the orthotropic one.
\end{corollary}

\section{Forms with trigonal symmetry}
\label{Sec:4}
\setcounter{equation}{0}

\subsection{Review and some theorems}
\label{Sub:4.1}

The first class that is not orthotropic-reducible is trigonal. Trigonal elasticity tensors characterize materials with threefold rotational symmetry, containing six independent elastic constants. In the Voigt notation, the linear elastic tensor has the form 
\begin{equation}
\label{4.1}
\BC=
\begin{bmatrix}
C_{11} & C_{12} & C_{13} & C_{14} & 0 & 0\\
C_{12} & C_{11} & C_{13} & -C_{14} & 0 & 0\\
C_{13} & C_{13} & C_{33} & 0 & 0 & 0\\
C_{14} & -C_{14} & 0 & C_{44} & 0 & 0\\
0 & 0 & 0 & 0 & C_{44} & C_{14}\\
0 & 0 & 0 & 0 &  C_{14} & \frac{C_{11}-C_{12}}{2}
\end{bmatrix}.
\end{equation}
The coupling $C_{14}$ is the first irreducible coupling in the usual hierarchy of three-dimensional crystallographic elasticity classes. Denote as before by $\BB$ the $3\times 3$ left upper block of $\BC$ and $C_{66}=\frac{C_{11}-C_{12}}{2}.$ The associated biquadratic in the variables $u,v$ and $w$ is given by:
\begin{equation} 
\label{4.2}
f(x\otimes y)=u^T\BB u+C_{44}(v_1^2+v_2^2)+C_{66}v_3^2+ 2C_{14}\big(u_1v_1-u_2v_1+v_2v_3\big).
\end{equation}
where the coupling $C_{14}$ links the shear $v_1$ to $u_1-u_2$ and $v_2$ to $v_3$. The acoustic $y$-tensor of $f$ satisfies 
$\BA(y)\succeq0$ for all $y\in\mathbb R^3.$ Note also, that rank-one convexity tested at $e_1\otimes e_2$ and $e_2\otimes e_3$ forces $C_{44},C_{66}\ge0$. 
\begin{lemma}
\label{Lem:4.1}
If a trigonal tensor is rank-one convex and not polyconvex, then
$$
 C_{44}>0,\quad C_{66}>0,\quad C_{14}\neq 0.
$$
\end{lemma}

\begin{proof}
If $C_{14}=0$, the tensor is orthotropic and is polyconvex by Theorem~\ref{Thm:3.1}. If $C_{66}=0$, positivity of
$\BA(e_1)$ forces $C_{14}=0$. If $C_{44}=0$, then
$\BA(e_3)=\operatorname{diag}(0,0,C_{33})$ has rank at most one, hence the
acoustic biquadratic has infinitely many zeros on $\mathbb P^2\times \mathbb P^2.$ By a theorem of Quarez 
[\ref{bib:Quarez}, Thm.~5.4], it is a sum of squares and therefore polyconvex.
\end{proof}

Since we are interested in nontrivial extremals, we assume in what follows that $C_{44},C_{66}>0$ and $C_{14}\neq 0$. Next we prove, that the rank-one convexity-polyconvexity gap for trigonal symmetry requires an indefinite normal block $\BB.$

\begin{theorem}[Indefinite normal block]
\label{Thm:4.2}
If $\BB$ is positive semi-definite, then rank-one convexity implies polyconvexity, both
holding if and only if $C_{44}C_{66}\ge C_{14}^2$. 
\end{theorem}

\begin{proof} The acoustic tensor at $e_1$ is 
$$
\BA(e_1)=C_{11}
\oplus 
\begin{bmatrix}C_{66}&C_{14}\\
C_{14}&C_{44}
\end{bmatrix},
$$
thus rank-one convexity forces $C_{44}C_{66} \geq C_{14}^2$. As $\BC$ is symmetric, it is sufficient to show that the generalized Schur complement of $\BB$ in $\BC$ is positive semidefinite as well.  Denote by $\BE, \BF,$ and $\BG$ the top right, bottom right, and bottom left $3\times 3$ blocks of $\BC$ respectively, so that the generalized Schur complement of $\BB$ in $\BC$ equals $\BC/\BB=\BF-\BG\BB^{\dagger}\BE,$ where $\BB^{\dagger}$ is the Moore-Penrose inverse of $\BB.$ Given the forms of $\BG$ and $\BE,$ calculating $\BC/\BB$ boils down to calculating $v^T\BB^{\dagger}v$ for $v=(1,-1,0).$  Since $\BB v=(C_{11}-C_{12})v=2C_{66}v,$ we have due to the symmetry of $\BB,$ that $v^T\BB=v^T\BB^T=(\BB v)^T=2C_{66}v^T.$ Consequenty, utilizing the identity $\BB\BB^{\dagger}\BB=\BB,$ we have on one hand $v^T\BB\BB^{\dagger}\BB v=v^T\BB v=4C_{66},$ and on the other hand $v^T\BB\BB^{\dagger}\BB v=2C_{66}v^T\BB^{\dagger}2C_{66}v=4C_{66}^2v^T\BB^{\dagger}v,$ hence we obtain  
$v^T\BB^{\dagger}v=1/C_{66}.$ Now a straightforward calculation gives
$$
\BC/\BB=(C_{44}-C_{14}^2/C_{66})
\oplus 
\begin{bmatrix}C_{44}&C_{14}\\
C_{14}&C_{66}
\end{bmatrix},$$ 
which is positive semidefinite when $C_{44}C_{66}\geq C_{14}^2$ and $C_{66}>0.$ This yields positive semidefiniteness of $\BC,$ thus polyconvexity of $f$. 

\end{proof}

\begin{remark}
\label{Rem:3.3}
If $\BB$ is positive semi-definite, then rank-one convexity of $f$ implies convexity of $f(x\otimes y)$ in the variables $u_i$ and $v_i$ only. 
\end{remark}

\subsection{A family of extreme rays}
\label{Sub:4.2}

With the use of AI engine Claude, we have found a family of extreme rays in the trigonal class. Below we provide the family with a proof of extremality employing the extremal determinant criterion. 

\begin{theorem}[Family of extremals]
\label{Thm:4.4}
Assume $k>0$ and define the trigonal tensor $\BC$ as
\begin{equation} 
\label{4.3}
(C_{11},C_{12},C_{13},C_{33},C_{44},C_{14})=\left(4,3,-\frac{16k^2}{9},0,k^2,\frac{2k}{3}\right),
 \quad C_{66}=\frac12.
\end{equation}
Then $\BC$ is rank-one convex but not polyconvex. Its acoustic tensor determinant has exactly ten distinct real projective zeros. Consequently, $\BC$ is a one-parameter family of extreme rays in the convex cone of $3\times 3$ quasiconex quadratic forms. 

\end{theorem}

\begin{proof}

\noindent\textbf{The biquadratic form.}
Direct calculation gives
\begin{align*}
 f_k(x\otimes y)={}&4u_1^2+4u_2^2+6u_1u_2
 -\frac{32k^2}{9}(u_1u_3+u_2u_3)\\
 &+k^2(v_1^2+v_2^2)+\frac12v_3^2
 +\frac{4k}{3}\bigl[s_1(u_1-u_2)+v_2v_3\bigr].
\end{align*}
The $y-$acoustic tensor is
\begin{equation}
\label{4.4}
\small \BA_k(y)=
\begin{bmatrix}
4y_1^2+\frac12y_2^2+k^2y_3^2+\frac{4k}{3}y_2y_3
&\frac72y_1y_2+\frac{4k}{3}y_1y_3
&\frac{4k}{3}y_1y_2-\frac{7k^2}{9}y_1y_3
\\[4mm]
\frac72y_1y_2+\frac{4k}{3}y_1y_3
&\frac12y_1^2+4y_2^2+k^2y_3^2-\frac{4k}{3}y_2y_3
&\frac{2k}{3}(y_1^2-y_2^2)-\frac{7k^2}{9}y_2y_3
\\[4mm]
\frac{4k}{3}y_1y_2-\frac{7k^2}{9}y_1y_3
&\frac{2k}{3}(y_1^2-y_2^2)-\frac{7k^2}{9}y_2y_3
&k^2(y_1^2+y_2^2)
\end{bmatrix}.
\end{equation}
\noindent\textbf{Determinant and positivity.}
Set
$$
 \rho^2=y_1^2+y_2^2,
 \qquad q=y_2(3y_1^2-y_2^2),
 \qquad z=ky_3.
$$
A direct calculation gives
\begin{equation}
\label{4.5}
\det \BA_k(y)=\frac{2k^2}{243}
\left(
27\rho^6+189q^2+72\rho^4z^2-432q\rho^2z
+48\rho^2z^4+224qz^3
\right).
\end{equation}
If $\rho=0$, this determinant is zero.  For $\rho>0$, write
$$
 r=\frac{z}{\rho},
 \qquad s=\frac{q}{\rho^3}.
$$
For $y_1=\rho\cos\phi$ and $y_2=\rho\sin\phi$ we have $s=3\sin\phi-4\sin^3\phi=\sin3\phi$, thus $|s|\leq1$. 
Equation (\ref{4.5}) simplifies to
$$
 \det A_k(x)=\frac{2k^2\rho^6}{243}F(r,s),
$$
where
$$
 F(r,s)=189s^2+(224r^3-432r)s+48r^4+72r^2+27.
$$
For fixed $r$, this is a convex quadratic function in $s\in [-1,1]$. Its endpoint values are
\begin{align*}
 F(r,1)&=8(r+3)^2\left(6\left(r-\frac23\right)^2+\frac13\right)\ge0,\\
 F(r,-1)&=8(r-3)^2\left(6\left(r+\frac23\right)^2+\frac13\right)\ge0.
\end{align*}
Its global minimizer is the vertex
$$
 s_0(r)=\frac{8r(27-14r^2)}{189},
$$
and completing the square gives
$$
 F(r,s)=189\bigl(s-s_0(r)\bigr)^2
 +\frac{(63-16r^2)(28r^2-9)^2}{189}.
$$
If $r^2\leq 63/16$ then $F(r,s)\geq 0$. If otherwise $r^2>63/16$, then we have
$$
 |s_0(r)|=\frac{8|r|(14r^2-27)}{189}
 \geq \frac{25\sqrt7}{28}>1,
$$
hence $F(r,s)\geq \min (F(r,1),F(r,1))\geq 0.$ Therefore we conclude that 
\begin{equation}
\label{4.6}
\det \BA_k(y)\ge0\quad\text{for all }\quad y\in\R^3.
\end{equation}
Note, next that from (\ref{4.4}), we have the $y-$acoustic tensor at $e_1$ is
\begin{equation}
\label{4.7}
\small \BA_k(e_1)=
\begin{bmatrix}
4 & 0  & 0 \\[1.5mm]
0 & \frac12 & \frac{2k}{3} \\[1.5mm]
0 &\frac{2k}{3} & k^2
\end{bmatrix},
\end{equation}
thus $\BA_k(e_1)\succ0$ with three strictly positive eigenvalues, therefore by (\ref{4.6}) and Proposition~\ref{Pro:2.3}, we get 
$\BA_k(e_1)\succeq0$ for all $y\in\R^3,$ i.e., $f_k$ is rank-one convex.  

\noindent \textbf{The ten real projective zeros.}
If $\rho=0$, the unique projective zero is the pole $P=e_3.$ For $\rho>0$, the endpoint factorizations show that endpoint zeros occur
only at $(r,s)=(3,-1)$ or equivalently, $(-3,1)$. They are (can be checked by direct calculation):
$$
 E_1=(0,2k,6),\qquad E_2=(-\sqrt3k,-k,6),\qquad E_3=(\sqrt3k,-k,6).
$$
Next, an interior zero must satisfy $s=s_0(r)$ (because the acoustic determinant is globally non-negative). The completed-square identity then gives
$$
 r^2=\frac9{28}\quad\text{or}\quad r^2=\frac{63}{16}.
$$
The second value is inadmissible because it would yield $|s_0(r)|>1$.  Thus, up to
$(r,s)\sim(-r,-s)$,
$$
 r=\frac{3}{2\sqrt7},
 \qquad s=\frac{10}{7\sqrt7}.
$$
The resulting six zeros are
$$
\begin{aligned}
 V_1&=(2\sqrt3k,4k,3),
&V_2&=(-3\sqrt3k,k.3),
&V_3&=(\sqrt3k,-5k,3),\\
 V_4&=(-2\sqrt3k,4k,3),
&V_5&=(-\sqrt3k,-5k,3),
&V_6&=(3\sqrt3k,k,3).
\end{aligned}
$$
Consequently
\begin{equation}
\label{4.7}
 Z_{\R}(\det \BA_k) =\{P,E_1,E_2,E_3,V_1,V_2,V_3,V_4,V_5,V_6\},
\end{equation}
so $\det \BA_k(y)$ has exactly $10$ distinct real projective zeros. By the work of Kunert and Scheiderer [\ref{bib:Kun.Sch.}], such  ternary sextics are extremals in the cone of nonnegative ternary sextics, thus so is $\det \BA_k(y).$ Note, that $\det \BA_k(y)$
is not a perfect square, because otherwise it would have infinitely many real projective zeros. Finally, by a theorem of Hovsepyan and Harutyunyan [\ref{bib:Har.Hov.}, Therem~2.1] we obtain that $f_k$ is an extreme ray in the full cone of $3\times 3$ quasiconvex quadratic form (not only the trigonal cone). Apparently $f_k$ is not polyconvex, because it is an extremal and is not convex. This completes the proof of the theorem.  

\end{proof}

\section*{Acknowledgements}
The work of D.H. is supported by the National Science Foundation under Grants No. DMS-2206239.


\begin{thebibliography}{99}

\bibitem{2} E. Acerbi and N. Fusco: Semicontinuity problems in the calculus of variations, \textit{Arch. Ration. Mech. Anal.}, 
86 (1984), 125--145.
\label{bib:Ace.Fus.}

\bibitem{2} G. Allaire and R.V. Kohn. Optimal lower bounds on the elastic energy of a composite made from two non-well-ordered isotropic materials, \textit{Quarterly of applied mathematics, } vol. LII, 311-333 (1994)
\label{bib:All.Koh.}

\bibitem{3}  J. M. Ball. Convexity conditions and existence theorems in nonlinear elasticity, \textit{Arch. Ration. Mech. Anal,} 63, 337-403 (1976).
\label{bib:Ball}

\bibitem{3}  O. Boussaid, C. Kreisbeck, and A. Schl\"omerkemper. Characterizations of Symmetric Polyconvexity, \textit{Arch. Ration. Mech. Anal,} Volume 234, pages 417--451 (2019).
\label{bib:Bou.Kre.Sch.}


\bibitem{5} G. Blekherman, J. Hauenstein, J. C. Ottem, K. Ranestad, B. Sturmfels, Algebraic boundaries of Hilbert's SOS cones, \textit{Compos. Math.,}148 (2012), 1717--1735. 
\label{bib:BLe.Hau.Ott.Ran.Stu.}


\bibitem{5}  A. V. Cherkaev and L. V. Gibiansky. The exact coupled bounds for effective tensors of electrical and magnetic properties of two-component two-dimensional composites, \textit{Proceedings of the Royal Society of Edinburgh. Section A, Mathematical and Physical Sciences,} 122 (1992), pp. 93--125.
\label{bib:Che.Gib.1}

\bibitem{5} A. V. Cherkaev and L. V. Gibiansky. Coupled estimates for the bulk and shear moduli of a two-dimensional isotropic elastic composite, \textit{Journal of the Mechanics and Physics of Solids,} 41 (1993), pp. 937--980.
\label{bib:Che.Gib.2}

\bibitem{milton-cherkaev-1995} G.~W.~Milton, A.~V.~Cherkaev, Which elasticity tensors are realizable?, \textit{J. Eng. Mater. Technol.,} 117 (1995), 483--493. 
\label{bib:Mil.Che.}

\bibitem{6} M. D. Choi, Positive semidefinite biquadratic forms, \textit{Linear Algebra Appl.}\ 12 (1975), 95--100. 
\label{bib:Choi}

\bibitem{6} M. D. Choi, T. Y. Lam, Extremal positive semidefinite forms, Math.\ Ann.\ 231 (1977), 1--18. 
\label{bib:Cho.Lam.}

\bibitem{6} M. D. Choi, T. Y. Lam, B. Reznick, Real zeros of positive semidefinite forms. I, Math.\ Z.\ 171 (1980), 1--26. \label{bib:Cho.Lam.Rez.}

\bibitem{6} B. Dacorogna. \textit{Direct methods in the calculus of variations.} Springer Applied Mathematical Sciences, Vol. 78, 2nd Edition (2008).
\label{bib:Dacorogna}

\bibitem{10} D. Harutyunyan, A note on the extreme points of the cone of quasiconvex quadratic forms with orthotropic symmetry, J.\ Elasticity 140 (2020), 79--93. 
\label{bib:Harutyunyan}

\bibitem{10} D. Harutyunyan, N. Hovsepyan, On the extreme rays of the cone of $3\times3$ quasiconvex quadratic forms: extremal determinants versus extremal and polyconvex forms, Arch.\ Ration.\ Mech.\ Anal.\ 244 (2022), 1--25. 
\label{bib:Har.Hov.}

\bibitem{10} D. Harutyunyan, G. W. Milton, Explicit examples of extremal quasiconvex quadratic forms that are not polyconvex, Calc.\ Var.\ Partial Differential Equations 54 (2015), 1575--1589. 
\label{bib:Har.Mil.1}

\bibitem{10} D. Harutyunyan, G. W. Milton, On the relation between extremal elasticity tensors with orthotropic symmetry and extremal polynomials, Arch.\ Ration.\ Mech.\ Anal.\ 223 (2017), 199--212. 
\label{bib:Har.Mil.2}

\bibitem{10} D. Harutyunyan, G. W. Milton, Towards characterization of all $3\times3$ extremal quasiconvex quadratic forms, Comm.\ Pure Appl.\ Math.\ 70 (2017), no.~11, 2164--2190. 
\label{bib:Har.Mil.3}


 \bibitem{12} J. Hou, Ch.-K. Li, Y.-T. Poon, X. Qi, and  N.-S. Sze. A new criterion and a special class of $k-$positive maps, \textit{Linear Algebra and its Applications} 470 (2015) 51-69.
 \label{bib:Hou.Li.Poo.Qi.Sze.}
 
 
 \bibitem{7}  H. Kang, E. Kim, and G. W. Milton.
Sharp bounds on the volume fractions of two materials in a two-dimensional body from electrical boundary measurements: the translation method,
\textit{Calculus of Variations and Partial Differential Equations,} 45, 367-401 (2012).
\label{bib:Kan.Kim.Mil.}

\bibitem{8} H. Kang and G. W. Milton. Bounds on the volume fractions of two materials in a
three dimensional body from boundary measurements by the translation method, \textit{SIAM Journal on Applied Mathematics,} 
73, 475--492 (2013).
 \label{bib:Kan.Mil.}

\bibitem{9} H. Kang, G. W. Milton, and J.-N. Wang.  Bounds on the Volume Fraction of the Two-Phase Shallow Shell Using One Measuremen, \textit{Journal of Elasticity,} 114, 41-53 (2014).
\label{bib:Kan.Mil.Wan.}


\bibitem{11} R. V. Kohn and R. Lipton. Optimal bounds for the effective energy of a mixture
of isotropic, incompressible, elastic materials.
\textit{Archive for Rational Mechanics and Analysis,} 102, 331--350 (1988).
\label{bib:Koh.Lip.}


\bibitem{12} A. Kunert, C. Scheiderer, Extreme positive ternary sextics, Trans.\ Amer.\ Math.\ Soc.\ 370 (2018), no.~6, 3997--4013; \textit{arXiv}:1508.03816. 
\label{bib:Kun.Sch.}

\bibitem{} X. Li and W. Wu. A class of generalized positive linear maps on matrix algebras, \textit{Linear algebra and its applications} 439 (2013) 2844-2860. 
\label{bib:Li.Wu.}

\bibitem{9} K. A. Lurie and A. V. Cherkaev. Accurate estimates of the conductivity of mixtures formed of two materials in a given proportion (two-dimensional problem), \textit{Doklady Akademii Nauk SSSR,} 264 (1982), 1128--1130, English translation in \textit{Soviet Phys. Dokl.,} 27 (1982), 461--462.
\label{bib:Lur.Che.1}

\bibitem{10} K. A. Lurie and A. V. Cherkaev. Exact estimates of conductivity of composites formed by two isotropically conducting media taken in prescribed proportion \textit{Proceedings of the Royal Society of Edinburgh, Section A, Mathematical and Physical Sciences} 99 (1984), 71--87.
\label{bib:Lur.Che.2}

\bibitem{10} G. W. Milton, A. V. Cherkaev, Which elasticity tensors are realizable?, J.\ Eng.\ Mater.\ Technol.\ 117 (1995), 483--493.
 \label{bib:Mil.Che.}
 
 
  \bibitem{12} G. W. Milton. On characterizing the set of positive effective tensors of composites: The variational method and the translation method, \textit{Communications on Pure and Applied Mathematics,} Vol. XLIII, 63--125 (1990).
\label{bib:Milton.1}

 \bibitem{13} G. W. Milton. \textit{The Theory of Composites,} vol. 6 of Cambridge Monographs on Applied and Computational Mathematics, Cambridge University Press, Cambridge, United Kingdom, 2002
\label{bib:Milton.2}

 \bibitem{14} G. W. Milton. Sharp inequalities which generalize the divergence theorem: an extension of the notion of quasi-convexity, \textit{Proceedings Royal Society A} 469, 20130075 (2013).
 \label{bib:Milton.3}

\bibitem{15} G. W. Milton and L. H. Nguyen. Bounds on the volume fraction of 2-phase, 2-dimensional
elastic bodies and on (stress, strain) pairs in composites.
\textit{Comptes Rendus M\'ecanique,} 340, 193-204 (2012).
\label{bib:Mil.Ngu.}


 \bibitem{16} C. B. Morrey. Quasiconvexity and the lower semicontinuity of multiple integrals,
 \label{bib:Morrey.1}               \textit{Pacific Journal of Mathematics} 2, 25--53 (1952).

\bibitem{17} C. B. Morrey. \textit{Multiple integrals in the calculus of variations},
  \label{bib:Morrey.2}             Springer--Verlag, Berlin, 1966.
  
  
\bibitem{17} S. M\"uller. Variational models for microstructure and phase transitions. \textit{In Calculus of variations and geometric
evolution problems (Cetraro, 1996),} volume 1713 of \textit{Lecture Notes in Math.,} pp. 85--210. Springer, Berlin, 1999.
\label{bib:Mueller}


 \bibitem{18} F. Murat and L. Tartar. Calcul des variations et homog\'en\'isation. (French) [Calculus of variation and homogenization], in
            Les m\'ethodes de l'homog\'en\'eisation: th\'eorie et applications en physique, volume 57 of Collection de la Direction des
\'etudes et recherches d'Electricit\'e de France, pages 319-369, Paris, 1985, Eyrolles, English
translation in Topics in the Mathematical Modelling of Composite Materials, pages 139-173, ed. by A. Cherkaev and R. Kohn, ISBN 0-8176-3662-5.
\label{bib:Mur.Tar.}

 \bibitem{18} V. Nesi and E. Rogora. A complete characterization of invariant jointly rank-$r$ convex quadratic forms and applications to composite materials,  \textit{ESAIM: Control, Optimization and the Calculus of Variations,}  (2007), Volume: 13, Issue: 1, page 1--34.
  \label{bib:Nes.Rog.}
   

\bibitem{20} R. Quarez, On the real zeros of positive semidefinite biquadratic forms, \textit{Comm. Algebra} 43 (2015), no.~3, 1317--1353. 
\label{bib:Quarez}

\bibitem{20} R.T.Rockafellar, Convex Analysis, \textit{Princeton University Press}, Princeton, NJ, 1970. 
\label{bib:Rockafellar}

\bibitem{20} B. Reznick, On Hilbert's construction of positive polynomials, \textit{arXiv}:0707.2156. 
\label{bib:Reznick}


\bibitem{20} M. Sion, On general minimax theorems, Pacific J.\ Math.\ 8 (1958), 171--176.
 \label{bib:Sion}
 
 
 \bibitem{19}  V. \v{S}ver\'ak. Rank-one convexity does not imply quasiconvexity. \textit{Proc. Roy. Soc. Edinburgh Sect. A,} 120(1-2):185--189, 1992.
\label{bib:Sverak}


\bibitem{18} L. Tartar, Estimation de coefficients homog{\'e}n{\'e}is{\'e}s ({French}) [{Estimation} of homogenization
                 coefficients], in Computing Methods in Applied Sciences and Engineering:
                 Third International Symposium, Versailles, France,
                 December 5--9, 1977, pages 364--373, Springer-Verlag, Berlin, 1979, English
                translation in Topics in the Mathematical Modelling of Composite Materials, pp. 9-20,
           ed. by A. Cherkaev and R. Kohn, ISBN 0-8176-3662-5.
\label{bib:Tartar.1}

\bibitem{19} L. Tartar, Estimations fines des coefficients homog\'en \'eis \'es.
          (French) [Fine estimations of homogenized coefficients], in Ennio de Giorgi Colloquium: Papers Presented at a
Colloquium held at the H. Poincar\'e Institute in November 1983, edited by P. Kr\'ee,
volume 125 of Pitman Research Notes in Mathematics, pages 168-187, London, 1985, Pitman Publishing Ltd.
  \label{bib:Tartar.2}   
  
  
  \bibitem{20} F. J. Terpstra, Die Darstellung biquadratischer Formen als Summen von Quadraten mit Anwendung auf die Variationsrechnung, Math.\ Ann.\ 116 (1938), 166--180. 
\label{bib:Terpstra}

\bibitem{20} L. Van Hove, Sur l'extension de la condition de Legendre du calcul des variations aux int\'egrales multiples \`a plusieurs fonctions inconnues, Nederl.\ Akad.\ Wetensch.\ Proc.\ 50 (1947), 18--23. 
\label{bib:VanHove.1}

\bibitem{27} L. Van Hove. Sur le signe de la variation seconde des int\'egrales multiples
\'a plusieurs functions inconnues, \textit{Acad. Roy. Belgique Cl. Sci. M\'em. Coll.} \textbf{24} (1949), 68.
  \label{bib:VanHove.2}    
  
\bibitem{28} K. Zhang, The structure of rank-one convex quadratic forms on linear elastic strains, Proc.\ Roy.\ Soc.\ Edinburgh Sect.\ A 133 (2003), 213--224. 
\label{bib:Zhang}



\end{thebibliography}
\end{document}